\documentclass{amsart}
\usepackage{amsopn,amssymb,mathrsfs}

\newtheorem{thm}{Theorem}[section]
\newtheorem{cor}[thm]{Corollary}
\newtheorem{lem}[thm]{Lemma}
\newtheorem{prop}[thm]{Proposition}
\newtheorem{prob}[thm]{Problem}

\theoremstyle{definition}
\newtheorem{defn}[thm]{Definition}

\theoremstyle{remark}

\newcommand{\ball}[1]{\ensuremath{B_{#1}}}
\newcommand{\bidual}[1]{\ensuremath{{#1}^{**}}}

\newcommand{\cderiv}[2]{\ensuremath{#1}^{({#2})}}

\newcommand{\Ck}[1]{\ensuremath{\mathscr{C}({#1})}}
\newcommand{\closure}[1]{\ensuremath{\overline{{#1}}}}

\newcommand{\Czerok}[1]{\ensuremath{\mathscr{C}_0({#1})}}

\newcommand{\dual}[1]{\ensuremath{{#1}^*}}

\newcommand{\lint}[4]{\ensuremath{\int_{#1}^{#2}{#3}\:\mathrm{d}{#4}}}

\newcommand{\mapping}[3]{\ensuremath{{#1}:{#2}\longrightarrow{#3}}}

\newcommand{\moddot}{\ensuremath{|\cdot|}}
\newcommand{\nat}{\mathbb{N}}
\newcommand{\norm}[1]{\ensuremath{||{#1}||}}
\newcommand{\normdot}{\ensuremath{||\cdot||}}

\newcommand{\oneton}[2]{\ensuremath{{#1}_1,\ldots,{#1}_{#2}}}

\newcommand{\pnorm}[2]{\ensuremath{||{#1}||_{#2}}}
\newcommand{\pnormdot}[1]{\ensuremath{||\cdot||_{#1}}}
\newcommand{\ptrinorm}[2]{\ensuremath{|||{#1}|||_{#2}}}
\newcommand{\ptrinormdot}[1]{\ensuremath{|||\cdot|||_{#1}}}
\newcommand{\rat}{\mathbb{Q}}
\newcommand{\real}{\mathbb{R}}
\newcommand{\res}[2]{\ensuremath{\res({#1};{#2})}}
\newcommand{\restrict}[1]{\ensuremath{\!\!\upharpoonright_{#1}}}
\newcommand{\setcomp}[2]{\ensuremath{\{{#1}\;|\;\,{#2}\}}}

\newcommand{\sph}[1]{\ensuremath{S_{#1}}}

\newcommand{\trinorm}[1]{\ensuremath{|||{#1}|||}}
\newcommand{\trinormdot}{\ensuremath{|||\cdot|||}}

\newcommand{\weakstar}{\ensuremath{w^*}}
\newcommand{\wone}{\ensuremath{\omega_1}}

\DeclareMathOperator{\aspan}{span} %
\DeclareMathOperator{\card}{card} %

\begin{document}

\title{Gruenhage compacta and strictly convex dual norms}
\author{Richard J.\ Smith}
\address{Queens' College, Cambridge, CB3 9ET, United Kingdom}
\email{rjs209@cam.ac.uk}
\thanks{Some of this research was conducted during a visit to the
University of Valencia, Spain. The author is grateful to A.\
Molt\'{o} for the invitation and subsequent discussions. He also
wishes to thank S.\ Troyanski and V.\ Montesinos for interesting
discussions and remarks.}

\subjclass[2000]{Primary 46B03; Secondary 46B26}

\date{October 2007}

\keywords{Gruenhage space, rotund, strictly convex, norm, tree,
renorming theory, perfect image, continuous image}


\begin{abstract}
We prove that if $K$ is a Gruenhage compact space then
$\dual{\Ck{K}}$ admits an equivalent, strictly convex dual norm. As
a corollary, we show that if $X$ is a Banach space and $\dual{X} =
\closure{\aspan}^{\trinormdot}(K)$, where $K$ is a Gruenhage compact
in the $\weakstar$-topology and $\trinormdot$ is equivalent to a
coarser, $\weakstar$-lower semicontinuous norm on $\dual{X}$, then
$\dual{X}$ admits an equivalent, strictly convex dual norm. We give
a partial converse to the first result by showing that if $\Upsilon$
is a tree, then $\dual{\Czerok{\Upsilon}}$ admits an equivalent,
strictly convex dual norm if and only if $\Upsilon$ is a Gruenhage
space. Finally, we present some stability properties satisfied by
Gruenhage spaces; in particular, Gruenhage spaces are stable under
perfect images.
\end{abstract}

\maketitle

\section{Introduction and preliminaries}

In renorming theory, we determine the extent to which the norm of a
given Banach space can be modified, in order to improve the geometry
of the corresponding unit ball. Naturally, the structural theory of
Banach spaces plays an important part in this field but, in recent
times, there has been a move toward a more non-linear, topological
approach. This new outlook led to the solution of some long-standing
problems, as well as producing some completely unexpected results.

Recall that a norm $\normdot$ on a real Banach space $X$ is called
\textit{strictly convex}, or \textit{rotund}, if $\norm{x} =
\norm{y} = \frac{1}{2}\norm{x+y}$ implies $x = y$. We say that
$\normdot$ is \textit{locally uniformly rotund}, or \textit{LUR},
if, given a point $x$ and a sequence $(x_n)$ in the unit sphere
$\sph{X}$ satisfying $\norm{x + x_n} \rightarrow 2$, we have $x_n
\rightarrow x$ in norm. If $\normdot$ is a dual norm on $\dual{X}$
then $\normdot$ is called $\weakstar$\textit{-LUR} if, given $x$ and
$(x_n)$ as above, we have $x_n \rightarrow x$ in the
$\weakstar$-topology. For a dual norm, evidently LUR $\Rightarrow$
$\weakstar$-LUR $\Rightarrow$ strictly convex.

It turns out that, in some contexts, these ostensibly convex,
geometrical properties of the norm can be characterised relatively
simply in purely non-linear, topological terms. Given a compact,
Hausdorff space $K$, we denote the Banach space of continuous
real-valued functions on $K$ by $\Ck{K}$, and identify
$\dual{\Ck{K}}$ with the space of regular, signed Borel measures on
$K$. Raja proved that if $K$ is a compact space then $\dual{\Ck{K}}$
admits an equivalent, dual LUR norm if and only if $K$ is
$\sigma$\textit{-discrete} \cite{raja:02}; that is, $K$ is a
countable union of sets, each of which is discrete in its subspace
topology. Moreover, $\dual{\Ck{K}}$ admits an equivalent
$\weakstar$-LUR norm if and only if $K$ is descriptive
\cite{raja:03}; the definition of a descriptive compact space is
given below. Raja also proved that $\dual{X}$ admits an equivalent
$\weakstar$-LUR norm if and only if $\ball{\dual{X}}$ is descriptive
in the $\weakstar$-topology.

Regarding strictly convex norms, the authors of \cite{motz:06}
recently showed that $X$, which can be a dual space, admits an
equivalent, strictly convex, $\sigma(X,N)$-lower semicontinuous norm
if and only if the square $\ball{X}^2$ has a certain linear,
topological decomposition with respect to a given norming subspace
$N \subseteq \dual{X}$. In this paper, we examine what can be done
without the linearity, and without explicit reference to the square.
Using Gruenhage compacta, we obtain a sufficient condition for a
dual space $\dual{X}$ to admit an equivalent, strictly convex dual
norm. This condition covers all established classes of Banach space
known to be so renormable, including the duals of all \textit{weakly
countably determined}, or Va\v{s}\'{a}k, spaces. It also covers the
more general class of `descriptively generated' dual spaces,
introduced recently in \cite{or:04}.

We define descriptive compact spaces and related notions. All
topological spaces are assumed to be Hausdorff. A family of subsets
$\mathscr{H}$ of a topological space $X$ is called \textit{isolated}
if, given $H \in \mathscr{H}$, there exists an open set $U$ that
includes $H$ and misses every other element of $\mathscr{H}$; i.e.\
$\mathscr{H}$ is discrete in the union $\bigcup \mathscr{H}$. The
family $\mathscr{H}$ is called a \textit{network} for $K$ if, given
$t \in U$, where $U$ is open, there exists $H \in \mathscr{H}$ such
that $t \in H \subseteq U$. In other words, a network is a basis,
but without the requirement that its elements be open subsets.
Finally, we say that a compact space $K$ is \textit{descriptive} if
it has a network $\mathscr{H}$ that is $\sigma$\textit{-isolated};
that is, $\mathscr{H} = \bigcup_n \mathscr{H}_n$, where each
$\mathscr{H}_n$ is a isolated family.

The class of descriptive compact spaces is rather large. It includes
two classes of topological spaces that have featured prominently in
non-separable Banach space theory, namely \textit{Eberlein} and
\textit{Gul'ko} compacta; see, for example \cite{fabian:97}. It also
includes all $\sigma$-discrete compact spaces; in particular, all
compact $K$ such that the Cantor derivative $\cderiv{K}{\wone}$ is
empty, where $\wone$ is the least uncountable ordinal. More
information about descriptive compact spaces can be found in
\cite{or:04}.

More generally, we say that a topological space $X$ is
\textit{fragmentable} if there exists a metric $d$ on $K$, with the
property that given $\varepsilon > 0$ and non-empty $E \subseteq T$,
there is an open set $U$ such that $U \cap E$ is non-empty and the
$d$-diameter of $E \cap U$ does not exceed $\varepsilon$. General
fragmentable compact spaces are not particularly well-behaved from
the point of view of renorming. Indeed, since every scattered space
is fragmented by the discrete metric, the compact $\wone + 1$ is
fragmentable, and it is well-known that $\dual{\Ck{\wone + 1}}$ does
not admit an equivalent, strictly convex dual norm; see, for example
\cite[Theorem VII.5.2]{dgz:93}. On the other hand, if $\dual{X}$
does admit an equivalent, strictly convex dual norm, then
$\ball{\dual{X}}$ is fragmentable in the $\weakstar$-topology
\cite{ribarska:92}.

The class of Gruenhage compact spaces fits between those of
descriptive and fragmentable spaces.

\begin{defn}[Gruenhage \cite{gru:87}]
\label{gruenhage} A topological space $X$ is called a
\textit{Gruenhage} space if there exist families $(\mathscr{U}_n)_{n
\in \nat}$ of open sets such that given distinct $x,y \in X$, there
exists $n \in \nat$ and $U \in \mathscr{U}_n$ with two properties:
\begin{enumerate}
\item $U \cap \{x,y\}$ is a singleton;
\item either $x$ lies in finitely many $U^\prime \in
\mathscr{U}_n$ or $y$ lies in finitely many $U^\prime \in
\mathscr{U}_n$.
\end{enumerate}
\end{defn}

If we were to follow Gruenhage's definition to the letter, the
sequence $(\mathscr{U}_n)$ above would have to cover $X$ as well,
but this demand is not necessary as property (1) forces the sequence
to cover all points of $X$, with at most one exception. Gruenhage
calls such sequences $\sigma$\textit{-distributively point-finite
}$T_0$\textit{-separating covers} of $X$.

In the next section, we investigate the role of Gruenhage spaces in
renorming theory. In the third section, we give a partial converse
to Theorem \ref{dualrotund}, the principal result of the second
section, and, by virtue of examples, get some measure of the gap
between descriptive compact spaces and Gruenhage compact spaces. The
last section is devoted to proving certain stability properties of
the class of Gruenhage spaces and its subclass of compact spaces.

\section{Gruenhage compacta and renorming}

We shall say that a family $\mathscr{H}$ of subsets of a topological
space $X$ \textit{separates points} if, given distinct $x,y \in X$,
there exists $H \in \mathscr{H}$ such that $\{x,y\} \cap H$ is a
singleton. It should be noted that some authors demand more of point
separation, namely that $H$ can be chosen in such a way that
$\{x,y\} \cap H = \{x\}$.

The next proposition brings together some equivalent formulations of
Gruenhage's definition that will be of use to us.

\begin{prop}
\label{equivalences} Let $X$ be a topological space. The following
are equivalent.
\begin{enumerate}
\item $X$ is a Gruenhage space;
\item there exists a sequence $(A_n)$ of closed sets and a sequence
$(\mathscr{H}_n)$ of families such that $\bigcup_n \mathscr{H}_n$
separates points, and furthermore each element of $\mathscr{H}_n$ is
an open subset of $A_n$ and disjoint from every other element of
$\mathscr{H}_n$;
\item there exists a sequence $(\mathscr{U}_n)$ of families of open
subsets of $X$ and sets $R_n$, such that $\bigcup_n \mathscr{U}_n$
separates points and $U \cap V = R_n$ whenever $U,V \in
\mathscr{U}_n$ are distinct.
\end{enumerate}
\end{prop}

\begin{proof}
(1) $\Rightarrow$ (2) follows directly from \cite[Proposition
7.4]{stegall:91}. Suppose that (2) holds. To obtain (3), simply
define $R_n = K\backslash A_n$ and set $\mathscr{U}_n = \setcomp{H
\cup R_n}{H \in \mathscr{H}_n}$. Finally, if (3) holds, define
$\mathscr{V}_n = \{R_n\}$. Given distinct $x,y \in X$, there exists
$n$ and $U \in \mathscr{U}_n$ such that $\{x,y\} \cap U$ is a
singleton. Let us assume that $x \in U$. There are two cases. If $x
\in R_n$ then $y \notin R_n$ because $R_n \subseteq U$, thus
$\{x,y\} \cap R_n$ is a singleton and, since $\mathscr{V}_n$ is a
singleton, $x$ is in exactly one element of $\mathscr{V}_n$.
Alternatively, we assume that $x \notin R_n$. Then $x \in V \in
\mathscr{U}_n$ forces $V = U$. Hence $x$ is in exactly one element
of $\mathscr{U}_n$. This shows that $X$ is Gruenhage.
\end{proof}

The second formulation presented in the proposition above prompts
the following definition.

\begin{defn}
\label{regularsystem} Let $X$ be a topological space. We call
$(A_n,\mathscr{H}_n)$ a \textit{legitimate system} if $A_n$ and
$(\mathscr{H}_n)$ are as in Proposition \ref{equivalences}, part
(2). We say that $\mathscr{H} = \bigcup_n \mathscr{H}_n$ is the
\textit{union} of the system.
\end{defn}

The next result follows easily.

\begin{cor}
\label{desciptivegru} A descriptive compact space is Gruenhage.
\end{cor}

\begin{proof}
In \cite{raja:03}, Raja shows that if $K$ is a descriptive compact
space then there exists a legitimate system $(A_n,\mathscr{H}_n)$
such that its union $\mathscr{H}$ is a network for $K$.
\end{proof}

%
We will spend a little time preparing our legitimate systems for
battle. We can and do assume for the rest of this section that every
legitimate system $(A_n,\mathscr{H}_n)$, with union $\mathscr{H}$,
satisfies three properties:

\begin{enumerate}
\item $\mathscr{H}$ is closed under the taking of finite intersections;
\item $K\backslash A_n \in \mathscr{H}$ for all $n$;
\item $A_n \backslash \bigcup \mathscr{H}_n \in \mathscr{H}$ for all $n$.
\end{enumerate}
Indeed, we first extend the system $(A_n,\mathscr{H}_n)$ by adding
the pairs $(K,\{K\backslash A_n\})$ and $(A_n \backslash \bigcup
\mathscr{H}_n, \{A_n \backslash \bigcup \mathscr{H}_n\})$ for every
$n$. We denote the extended system again by $(A_n,\mathscr{H}_n)$
and then consider, for each non-empty, finite $F \subseteq \nat$,
the pairs $(A_F,\mathscr{H}_F)$, where $A_F = \bigcap_{n \in F} A_n$
and
$$
\mathscr{H}_F \;=\; \setcomp{\textstyle{\bigcap_{n \in F}} H_n}{H_n
\in \mathscr{H}_F}.
$$

A family $\mathscr{H}$ of pairwise disjoint subsets of $K$ is called
\textit{scattered} if there exists a well-ordering $(H_\xi)_{\xi <
\lambda}$ of $\mathscr{H}$ such that $\bigcup_{\xi < \alpha} H_\xi$
is open in $\bigcup \mathscr{H}$ for all $\alpha < \lambda$.
Equivalently, $\mathscr{H}$ is scattered if, given non-empty $M
\subseteq \bigcup \mathscr{H}$, there exists $H \in \mathscr{H}$
such that $M \cap H$ is non-empty and open in $M$. Scattered
families naturally generalise isolated ones. The following lemma is
a simple extension of Rudin's result that Radon measures on
scattered compact spaces are atomic. We can state it in greater
generality than required, without compromising the simplicity of the
proof. We will say that $H \subseteq K$ is \textit{universally Radon
measurable} (uRm) if, given positive $\mu \in \dual{\Ck{K}}$, there
exist Borel sets $E$, $F$ such that $E \subseteq H \subseteq F$ and
$\mu(E) = \mu(F)$; equivalently, $H$ can be measured by the
completion of each such $\mu$, which we again denote by $\mu$.

\begin{lem}
\label{rudinextension} If $\mathscr{H}$ is a scattered family of uRm
subsets of a compact space $K$ then $\bigcup \mathscr{H}$ is uRm and
$\mu(\bigcup \mathscr{H}) = \sum_{H \in \mathscr{H}} \mu(H)$ for
every positive $\mu \in \dual{\Ck{K}}$.
\end{lem}

\begin{proof}
Take a well-ordering $(H_\xi)_{\xi < \lambda}$ of $\mathscr{H}$ and
open sets $U_\alpha$, $\alpha < \lambda$, such that $H_\alpha
\subseteq U_\alpha$ and $U_\alpha \cap H_\beta$ is empty whenever
$\alpha < \beta$. We proceed by transfinite induction on $\lambda$;
note that by $\sigma$-additivity, we can assume that $\lambda$ is a
limit ordinal of uncountable cofinality. Set $D_\alpha = U_\alpha
\backslash \bigcup_{\xi < \alpha} U_\xi$ for $\alpha < \lambda$.
Given positive $\mu \in \dual{\Ck{K}}$, by the uncountable
cofinality, let $\alpha < \lambda$ such that $\mu(D_\beta) = 0$ for
$\alpha \leq \beta < \lambda$. The regularity of $\mu$ ensures that
$\mu(D) = 0$, where $D = \bigcup_{\alpha \leq \beta < \lambda}
D_\beta$. By inductive hypothesis, there exist Borel sets $E$, $F$
such that $E \subseteq \bigcup_{\xi < \alpha} H_\xi \subseteq F$ and
$\mu(E) = \mu(F) = \sum_{\xi < \alpha} \mu(H)$, so the conclusion
follows when we consider $E$ and $F \cup D$.
\end{proof}

It is evident that, given a legitimate system $(A_n,\mathscr{H}_n)$,
the family $\mathscr{H}^\prime_n = \mathscr{H}_n \cup \{K\backslash
A_n, A_n\backslash \bigcup \mathscr{H}_n\}$ is scattered and has
union $K$. Moreover, the family
$$
\mathscr{D}_n = \textstyle{\setcomp{\bigcap_{i \leq n} H_1 \cap
\ldots \cap H_n}{H_i \in \mathscr{H}^\prime_i}}
$$
also enjoys these properties. Readers familiar with related
literature will recognise that these families lead directly to
fragmentability, via Ribarska's characterisation of fragmentable
spaces \cite{ribarska:87}. Elements of the proof of the following
result appear in \cite{stegall:91}. We denote both canonical norms
on $\Ck{K}$ and $\bidual{\Ck{K}}$ by $\pnormdot{\infty}$, and that
of $\dual{\Ck{K}}$ by $\pnormdot{1}$. We will be identifying certain
subsets of $K$ with their indicator functions, either in $\Ck{K}$ or
$\bidual{\Ck{K}}$.

\begin{lem}
\label{norming} Let $(A_n,\mathscr{H}_n)$ be a legitimate system
that separates points, with union $\mathscr{H}$ and which satisfies
properties (1) -- (3) above. Then $N =
\overline{\aspan}^{\pnormdot{\infty}}(\mathscr{H})$ is a subalgebra
of $\bidual{\Ck{K}}$ that is 1-norming for $\dual{\Ck{K}}$.
\end{lem}

\begin{proof}
Let $\mathscr{D}_n$ be the families introduced above, with union
$\mathscr{D}$. As $\mathscr{H}$ separates points, so does
$\mathscr{D}$. If $\mu \in \dual{\Ck{K}}$ has variation $|\mu|$ then
we have $\pnorm{\mu}{1} = \sum_{D \in \mathscr{D}_n} |\mu|(D)$ by
Lemma \ref{rudinextension}. Thus, given $\varepsilon > 0$, we can
take finite subsets $\mathscr{F}_n \subseteq \mathscr{D}_n$ and
compact subsets $K_D \subseteq D$, $D \in \mathscr{F}_n$, such that
$\sum_n |\mu|(K\backslash \bigcup_{D \in \mathscr{F}_n} K_D) <
\varepsilon$. Put $M = \bigcap_n \bigcup_{D \in \mathscr{F}_n} K_D$
and $M_D = M \cap K_D = M \cap D$. If $\mathscr{M}_n =
\setcomp{M_D}{D \in \mathscr{F}_n}$ then $\mathscr{M}_n$ is family
of pairwise disjoint sets with union $M$, and $\mathscr{M}_{n+1}$
refines $\mathscr{M}_n$. Moreover, each $M_D$ is clopen in $M$ and,
as $\mathscr{D}$ separates points of $K$, so $\mathscr{M} =
\bigcup_n \mathscr{M}_n$ separates points of $M$. Therefore, by the
Stone-Weierstrass Theorem, $\Ck{M} =
\overline{\aspan}^{\pnormdot{\infty}}(\mathscr{M})$.

It follows that we can take non-empty, disjoint $M_{D_i} \in
\mathscr{M}$ and $a_i \in [1,-1]$, $i \leq n$, such that $|\mu|(M) -
\sum_{i \leq n} a_i \mu(M_{D_i}) < \varepsilon$. Now $M_{D_i},
M_{D_j} \neq \varnothing$ and $M_{D_i} \cap M_{D_j} = \varnothing$
implies $D_i \cap D_j = \varnothing$. Therefore, $\sum_{i \leq n}
|\mu|(D_i\backslash M_{D_i}) \leq |\mu|(K\backslash M) <
\varepsilon$. We conclude that $\pnorm{\mu}{1} - \sum_{i \leq n} a_i
\mu(D_i) < 2\varepsilon$. Since $\mathscr{D} \subseteq \mathscr{H}$,
we are done.
\end{proof}

We say that a norm $\normdot$ on $X$ is \textit{pointwise uniformly
rotund}, or \textit{p-UR}, if there exists a separating subspace $F
\subseteq \dual{X}$ such that, given sequences $(x_n)$ and $(y_n)$
satisfying $\norm{x_n} = \norm{y_n} = 1$ and $\norm{x_n + y_n}
\rightarrow 2$, then $f(x_n - y_n) \rightarrow 0$ for all $f \in F$;
see, for example \cite{rychtar:05}. Evidently, p-UR norms are
strictly convex. We can now present the main theorem.

\begin{thm}
\label{dualrotund} If $K$ is a Gruenhage compact then:
\begin{enumerate}
\item $\dual{\Ck{K}}$ admits an equivalent, strictly convex, dual lattice
norm;
\item $\dual{\Ck{K}}$ admits an equivalent, dual p-UR norm.
\end{enumerate}
\end{thm}

\begin{proof}
The lattice norm is constructed first. We take a legitimate system
$(A_n, \mathscr{H}_n)$ satisfying the conclusion of Lemma
\ref{norming}. For $\mu \in \dual{\Ck{K}}$ and $m \geq 1$, define
the seminorm
$$
\pnorm{\mu}{n,m}^2 \;=\; \inf \setcomp{m^{-1}\textstyle{\sum_{H \in
\mathscr{H}_n}} |\lambda|(H)^2 + \pnorm{\mu - \lambda}{1}^2}{\lambda
\in \dual{\Ck{A_n}}}.
$$
We observe that $\pnorm{\mu}{n,m} \leq \pnorm{\mu}{1}$ and that
$\pnormdot{n,m}$ is $\weakstar$-lower semicontinuous. We can verify
the lower semicontinuity by applying a compactness argument.
Alternatively, if we denote the open set $(\bigcup \mathscr{H}_n)
\cup (K\backslash A_n)$ by $U$, we observe that $\pnorm{\mu}{n,m} =
\sup{\setcomp{\mu(f)}{f \in B}}$, where
$$
B \;=\; \setcomp{f \in \Czerok{U}}{m\textstyle{\sum_{H \in
\mathscr{H}_n}} \pnorm{f\restrict{H}}{\infty}^2 +
\pnorm{f}{\infty}^2 \leq 1}.
$$
In this way, we see that $\pnormdot{n,m}$ is also a lattice
seminorm.

We define a dual lattice norm on $\dual{\Ck{K}}$ by setting
$$
\norm{\mu}^2 \;=\; \pnorm{\mu}{1}^2 + \sum_{n,m} 2^{-n-m}
\pnorm{\mu}{n,m}^2.
$$
Now suppose that $\norm{\mu} = \norm{\nu} = \frac{1}{2}\norm{\mu +
\nu}$. A standard convexity argument (cf.\ \cite[Fact
II.2.3]{dgz:93}) yields
\begin{equation}
\label{eqn1} 2\pnorm{\mu}{n,m}^2 + 2\pnorm{\nu}{n,m}^2 - \pnorm{\mu
+ \nu}{n,m}^2 = 0
\end{equation}
for all $n$ and $m$. By appealing to compactness or the Hahn-Banach
Theorem, there exist $\mu_{n,m}, \nu_{n,m} \in \dual{\Ck{A_n}}$ such
that
$$
\pnorm{\mu}{n,m}^2 \;=\; m^{-1}\textstyle{\sum_{H \in
\mathscr{H}_n}} |\mu_{n,m}|(H)^2 + \pnorm{\mu - \mu_{n,m}}{1}^2
$$
and likewise for $\nu$. Hence, by applying further standard
convexity arguments to equation (\ref{eqn1}), we obtain
\begin{equation}
\label{eqn2} 2|\mu_{n,m}|(H)^2 + 2|\nu_{n,m}|(H)^2 - |\mu_{n,m} +
\nu_{n,m}|(H)^2 \;=\; 0
\end{equation}
for all $n$, $m$ and $H \in \mathscr{H}_n$. Now we estimate
\begin{eqnarray*}
\pnorm{\mu\restrict{A_n} - \mu_{n,m}}{1} &=& \pnorm{\mu -
\mu_{n,m}}{1} - \pnorm{\mu\restrict{K\backslash A_n}}{1} \\
&\leq& \pnorm{\mu}{n,m} - \pnorm{\mu\restrict{K\backslash A_n}}{1}
\\
&\leq& [m^{-1}\textstyle{\sum_{H \in \mathscr{H}_n}} |\mu|(H)^2 +
\pnorm{\mu \restrict{K\backslash A_n}}{1}^2]^{\frac{1}{2}} -
\pnorm{\mu\restrict{K\backslash A_n}}{1} \\
&\leq& m^{-\frac{1}{2}}
\end{eqnarray*}
because $\pnormdot{1} \leq \normdot$. A similar result holds for
$\nu$. Therefore, we conclude from equation (\ref{eqn2}) that
$$
2|\mu|(H)^2 + 2|\nu|(H)^2 - |\mu + \nu|(H)^2 \;=\; 0
$$
for all $H \in \mathscr{H}_n$ and $n \in \nat$. As $N$ from Lemma
\ref{norming} is norming, we certainly obtain $|\mu| = |\nu| =
\frac{1}{2}|\mu + \nu|$. This gives $\mu = \nu$ by the following
lattice argument, included for completeness. If $\lambda = \mu_+ -
\nu_-$ then $|\mu| = |\nu|$ implies $\lambda = \nu_+ - \mu_-$,
meaning $\mu + \nu = 2\lambda$. Hence $\mu_+ + \mu_- = |\mu| =
\frac{1}{2}|\mu + \nu| = \lambda_+ + \lambda_-$. We see that
$\lambda_+ = (\mu_+ - \nu_-)_+ \leq (\mu_+)_+ = \mu_+$, hence $\mu_+
= \lambda_+$ and $\mu_- = \lambda_-$. We conclude that $\mu = \nu$
as claimed.

Now we construct the p-UR norm, using the norming subspace $N$.
First, we claim that $\normdot$ above already satisfies the p-UR
property if $\mu_k$ and $\nu_k$ are positive. Suppose that $\mu_k$
and $\nu_k$ are positive measures such that $\norm{\mu_k} =
\norm{\nu_k} = 1$ and $\norm{\mu_k + \nu_k} \rightarrow 2$. As
above, we can find $\mu_{k,n,m}, \nu_{k,n,m} \in \dual{\Ck{A_n}}$
such that
$$
\pnorm{\mu}{k,n,m}^2 \;=\; m^{-1}\textstyle{\sum_{H \in
\mathscr{H}_n}} |\mu_{k,n,m}|(H)^2 + \pnorm{\mu_k -
\mu_{k,n,m}}{1}^2
$$
and likewise for $\nu_k$. By convexity arguments, we obtain
\begin{equation}
\label{eqn3} 2|\mu_{k,n,m}|(H)^2 + 2|\nu_{k,n,m}|(H)^2 -
|\mu_{k,n,m} + \nu_{k,n,m}|(H)^2 \rightarrow 0
\end{equation}
as $k \rightarrow \infty$. Moreover, if $H \in \mathscr{H}_n$, we
estimate
$$
|\mu_k - \mu_{k,n,m}|(H) \leq \pnorm{\mu_k \restrict{A_n} -
\mu_{k,n,m}}{1} \leq m^{-\frac{1}{2}}
$$
and likewise for $\nu_k$. Therefore, by fixing $m$ large enough and
appealing to equation (\ref{eqn3}), we get
$$
2\mu_k(H)^2 + 2\nu_k(H)^2 - (\mu_k + \nu_k)(H)^2 \rightarrow 0
$$
whence $(\mu_k - \nu_k)(H) \rightarrow 0$. It follows that
$\xi(\mu_k - \nu_k) \rightarrow 0$ for all $\xi \in N$, thus
completing the claim.

Now we set $\trinorm{\mu}^2 = \norm{\mu_+}^2 + \norm{\mu_-}^2$. To
see that this defines a dual norm, observe that as $\normdot$ is a
lattice norm, we have
$$
\norm{\mu_+} \;=\; \sup \setcomp{\mu(f)}{f \in \Ck{K}, f \geq 0
\mbox{ and }\norm{f} \leq 1}
$$
where $\normdot$ also denotes the predual norm. Thus $\mu \mapsto
\norm{\mu_+}$ is $\weakstar$-lower semicontinuous, and likewise for
$\mu \mapsto \norm{\mu_-}$. Now, given general $\mu_k$ and $\nu_k$
satisfying
$$
2\trinorm{\mu_k}^2 + 2\trinorm{\nu_k}^2 - \trinorm{\mu_k + \nu_k}^2
\rightarrow 0
$$
we get
$$
2\norm{(\mu_k)_+}^2 + 2\norm{(\nu_k)_+}^2 - \norm{(\mu_k)_+ +
(\nu_k)_+}^2 \rightarrow 0
$$
and similarly for $(\mu_k)_-$ and $(\nu_k)_-$. Therefore, we can
apply the claim twice to get $(\mu_k - \nu_k)(H) \rightarrow 0$ for
all $H \in \mathscr{H}$.
\end{proof}

We apply the theorem above to obtain renorming results for more
general Banach spaces. First, we give a modest generalisation of the
classic transfer method for LUR renormings, applied to strictly
convex renormings; cf.\ \cite[Theorem II.2.1]{dgz:93}. A proof is
provided for completeness.

\begin{prop}
\label{rotundtransfer} Let $\dual{(X,\normdot)}$,
$\dual{(Y,\normdot)}$ be dual Banach spaces, with
$\dual{(Y,\normdot)}$ strictly convex. Further, let $\trinormdot$ be
a coarser, $\weakstar$-lower semicontinuous seminorm on $\dual{X}$,
$\mapping{T}{\dual{X}}{\dual{Y}}$ a bounded, linear operator and set
$Z = \closure{\dual{T}\dual{Y}}^{\trinormdot} \subseteq \dual{X}$.
Then there exists an equivalent dual norm $\moddot$ on $X$, such
that whenever
$$
f \in Z,\;f^\prime \in X \quad\mbox{and}\quad |f| = |f^\prime| =
\textstyle{\frac{1}{2}}|f + f^\prime|
$$
we have $\trinorm{f - f^\prime} = 0$.
\end{prop}

\begin{proof}
Define seminorms $|\cdot|_n$ on $\dual{X}$ by
$$
|f|^2_n \;=\; \inf\setcomp{\trinorm{f - \dual{T}g} +
n^{-1}\norm{g}}{g \in \dual{Y}}
$$
and set $|f|^2 = \norm{f}^2 + \sum_{n \geq 1}2^{-n}|f|^2_n$. Since
$\trinormdot$ is coarser than $\normdot$, our new norm $\moddot$ is
equivalent to $\normdot$. As in Theorem \ref{dualrotund}, by a
$\weakstar$-compactness argument or the Hahn-Banach Theorem,
$|\cdot|_n$ is a $\weakstar$-lower semicontinuous seminorm, and the
infimum in the definition is attained. Now let $f$ and $f^\prime$
satisfy the above hypothesis. By convexity arguments and infimum
attainment, we can take $g_n, g^\prime_n \in \dual{Y}$ such that
\begin{equation}
\label{eqn4} |f|^2_n \;=\; \trinorm{f - \dual{T}g_n}^2 +
n^{-1}\norm{g_n}^2,
\end{equation}
\begin{equation}
\label{eqn5} \trinorm{f - \dual{T}g_n} \;=\; \trinorm{f^\prime -
\dual{T}g^\prime_n}
\end{equation} and
$$
\norm{g_n} \;=\; \norm{g^\prime_n} \;=\;
\textstyle{\frac{1}{2}}\norm{g_n + g^\prime_n}.
$$
The last equation tells us that $g_n = g^\prime_n$ for all $n$,
meaning that we have
$$
\trinorm{f - f^\prime} \;\leq\; \trinorm{f - \dual{T}g_n} +
\trinorm{f^\prime - \dual{T}g^\prime_n}
$$
Since $f \in Z$, we have $|f|_n \rightarrow 0$, so by equations
(\ref{eqn4}) and (\ref{eqn5}), this leads to $\trinorm{f^\prime -
\dual{T}g^\prime_n} = \trinorm{f - \dual{T}g_n} \rightarrow 0$,
giving $\trinorm{f - f^\prime} = 0$ as required.
\end{proof}

Using this, we can obtain our general renorming result.

\begin{prop}
\label{coarsegruenhage} Let $(X,\normdot)$ be a Banach space, $F
\subseteq \dual{X}$ a subspace and $\trinormdot$ a coarser norm on
$X$, such that $F \cap \dual{(X,\trinormdot)}$ separates points of
$X$. Further, let $K \subseteq X$ be a Gruenhage compact in the
$\sigma(X,F)$-topology and suppose $X =
\closure{\aspan}^{\trinormdot}(K)$. Then:
\begin{enumerate}
\item there is a coarser, $\sigma(X,F)$-lower semicontinuous, strictly convex norm $\moddot$ on $X$;
\item $X$ admits an equivalent, strictly convex norm.
\end{enumerate}
Moreover, if $F$ is a norming subspace then $\moddot$ is equivalent
to $\normdot$.
\end{prop}

\begin{proof}
Since $F$ is separating, we can identify
$((X,\normdot),\sigma(X,F))$ as a topological subspace of
$(\dual{(F,\normdot)},\weakstar)$ by standard evaluation and
consider $K$ as a $\weakstar$-compact subset of $\dual{F}$. Now
elements of $F$ act as continuous functions on $(K,\weakstar)$ and
the map $\mapping{S}{\dual{\Ck{K}}}{\dual{F}}$, given by $(S\mu)(f)
= \lint{K}{}{f}{\mu}$, is a dual operator. Let $\trinormdot$ also
denote the canonical norm on $G = \dual{(X,\trinormdot)}$, and
define the $\weakstar$-lower semicontinuous seminorm
$$
\ptrinorm{\xi}{1} \;=\; \sup\setcomp{\xi(f)}{f \in F \mbox{ and
}\trinorm{f} \leq 1}
$$
on $\dual{F}$. By Proposition \ref{rotundtransfer}, there exists an
equivalent, dual norm $\moddot_1$ on $\dual{F}$, such that if
$$
\xi \in \closure{S\dual{\Ck{K}}}^{\ptrinormdot{1}},\; \xi^\prime \in
\dual{F} \quad\mbox{and}\quad |\xi|_1 \;=\; |\xi^\prime|_1 \;=\;
\textstyle{\frac{1}{2}}|\xi + \xi^\prime|_1
$$
then $\ptrinorm{\xi - \xi^\prime}{1} = 0$.

Let $\moddot$ be the restriction of $|\cdot|_1$ to $X$ and note that
$\moddot$ is both $\sigma(X,F)$-lower semicontinuous and coarser
than $\normdot$. Moreover, $X = \closure{\aspan}^{\trinormdot}(K)
\subseteq \closure{S\dual{\Ck{K}}}^{\ptrinormdot{1}}$. Therefore,
whenever $|x| = |x^\prime| = \frac{1}{2}|x + x^\prime|$, we have
$\ptrinorm{x - x^\prime}{1} = 0$. Since $F \cap G$ separates points
of $X$, it follows that $x^\prime = x$. This gives (1). For (2),
observe that the sum $\normdot + \moddot$ is an equivalent, strictly
convex norm on $X$. Finally, if $F$ is norming then $\moddot$ is
equivalent to $\normdot$.
\end{proof}

Let us assume that the coarser norm $\trinormdot$ of Proposition
\ref{coarsegruenhage} is $\sigma(X,F)$-lower semicontinuous. By a
standard polar argument
$$
\trinorm{x} \;=\; \sup\setcomp{f(x)}{f \in F,\mbox{ }\trinorm{f}
\leq 1}
$$
and, in particular, $F \cap \dual{(X,\trinormdot)}$ separates points
of $X$.

\begin{cor}
\label{dualgrugen} Let $X$ be a Banach space and $\dual{X} =
\overline{\aspan}^{\trinormdot}(K)$, where $K$ is a Gruenhage
compact in the $\weakstar$-topology and $\trinormdot$ is equivalent
to a coarser, $\weakstar$-lower semicontinuous norm on $\dual{X}$.
Then $\dual{X}$ admits an equivalent, strictly convex dual norm.
\end{cor}

The result above applies to all established classes of Banach spaces
known to admit equivalent strictly convex dual norms on their dual
spaces; for example, Va\v{s}\'{a}k spaces. We move on to discuss a
property of Banach spaces, introduced in \cite{fmz:04} and shown
there to be a sufficient condition for the existence of an
equivalent, strictly convex dual norm.

\begin{defn}[\cite{fmz:04}]
\label{propertyg} We say that the Banach space $X$ has property G if
there exists a bounded set $\Gamma = \bigcup_{n \in \nat} \Gamma_n
\subseteq X$, with the property that whenever $f,g \in
\ball{\dual{X}}$ are distinct, there exist $n \in \nat$ and $\gamma
\in \Gamma_n$ such that $(f - g)(\gamma) \neq 0$ and, either
$|f(\gamma^\prime)|
> \frac{1}{4}|(f-g)(\gamma)|$ for finitely many $\gamma^\prime \in
\Gamma_n$, or $|g(\gamma^\prime)| > \frac{1}{4}|(f-g)(\gamma)|$ for
finitely many $\gamma^\prime \in \Gamma_n$.
\end{defn}

As well as showing that all Va\v{s}\'{a}k spaces possess property G,
the authors of \cite{fmz:04} remark that the property is closely
related to Gruenhage compacta.

\begin{prop}
\label{gimpliesdualgru} If $X$ has property G then the dual unit
ball $\ball{\dual{X}}$ is a Gruenhage compact in the
$\weakstar$-topology.
\end{prop}

\begin{proof}
We can and do assume that $\Gamma$ is a subset of the unit ball
$\ball{X}$. Given $\gamma \in \Gamma$ and $q \in (0,1) \cap \rat$,
we let $U(\gamma,q) = \setcomp{f \in \ball{\dual{X}}}{f(\gamma) >
q}$. We prove that, together, $(\mathscr{U}_{n,q})$ and
$(\mathscr{V}_{n,q})$, $n \in \nat$ and $q \in (0,1) \cap \rat$,
satisfy (1) and (2) of Definition \ref{gruenhage}, where
$\mathscr{U}_{n,q} = \setcomp{U(\gamma,q)}{\gamma \in \Gamma_n}$ and
$\mathscr{V}_{n,q} = \setcomp{-U(\gamma,q)}{\gamma \in \Gamma_n}$.
Given distinct $f,g \in \ball{\dual{X}}$, take $\gamma \in \Gamma_n$
with the property that $\alpha = \frac{1}{4}|(f-g)(\gamma)| > 0$. It
follows that either $|f(\gamma)|
> \alpha$ or $|g(\gamma)| > \alpha$; without loss of
generality, we assume that the former inequality holds. Now suppose
that $f(\gamma) > 0$. We choose rational $q$ to satisfy $f(\gamma) >
q
> \max \{g(\gamma),\alpha \}$ if $f(\gamma) >
g(\gamma)$, or $g(\gamma) > q > f(\gamma)$ otherwise. Either way,
$U(\gamma,q) \cap \{f,g\}$ is a singleton, giving (1). Since $q
> \alpha$, (2) follows. If $f(\gamma) < 0$, we repeat the above
argument with $-f$ and $-g$.
\end{proof}

\begin{cor}[\cite{fmz:04}]
If $X$ has property G then $\dual{X}$ admits an equivalent, strictly
convex dual norm.
\end{cor}

\begin{proof}
Combine Proposition \ref{gimpliesdualgru} and Corollary
\ref{dualgrugen}.
\end{proof}

We finish this section with an open problem.

\begin{prob}
If $\dual{\Ck{K}}$ admits a strictly convex dual norm then is $K$
Gruenhage? More ambitiously, if $\dual{X}$ is a dual Banach space
with strictly convex dual norm, is $\ball{\dual{X}}$ Gruenhage?
\end{prob}

\section{A topological characterisation of $Y$-embeddable trees}

In this section, we present a partial converse to Theorem
\ref{dualrotund}. We call a partially ordered set
$(\Upsilon,\preccurlyeq)$ a \textit{tree} if, for each $t \in
\Upsilon$, the set $(0,t] = \setcomp{s \in \Upsilon}{s \preccurlyeq
t}$ of \textit{predecessors} of $t$ is well-ordered. Given $t \in
\Upsilon$, we denote by $t^+$ the set of \textit{immediate
successors} of $t$ in $\Upsilon$; that is, $u \in t^+$ if and only
if $t \prec u$ and $t \prec \xi \prec u$ for no $\xi$. The locally
compact, scattered \textit{order topology} on $\Upsilon$ takes as a
basis the sets $(s,t]$, $s \prec t$, where $(s,t] =
(0,t]\backslash(0,s]$. To ensure that this topology is also
Hausdorff, we demand that every non-empty, totally ordered subset of
$\Upsilon$ has at most one minimal upper bound; trees satisfying
this property are themselves called \textit{Hausdorff}. We study the
space $\Czerok{\Upsilon}$ of continuous, real-valued functions on
$\Upsilon$ that vanish at infinity, and the dual space of measures.
To date, most of the results about renorming $\Czerok{\Upsilon}$ and
its dual have been order-theoretic in character: \cite{haydon:99},
\cite{smith:05b} and \cite{smith:06}. Such order-theoretic results,
while well-suited in this context, are deeply bound to the
tree-structure and, as such, do not offer obvious generalisations.
Here, we are able to give a purely topological characterisation of
trees $\Upsilon$, such that $\dual{\Czerok{\Upsilon}}$ admits an
equivalent, strictly convex dual norm. The following definition
first appears in \cite{smith:05b}.

\begin{defn}
\label{ordery} Let $Y$ be the set of all strictly increasing,
continuous, transfinite sequences $x = (x_\alpha)_{\alpha \leq
\beta}$ of real numbers, where $0 \leq \beta < \wone$. We order $Y$
by declaring that $x < y$ if and only if either $y$ strictly extends
$x$, or if there is some ordinal $\alpha$ such that $x_\xi = y_\xi$
for $\xi < \alpha$ and $y_\alpha < x_\alpha$.
\end{defn}

We say that a map $\mapping{\rho}{\Upsilon}{\Sigma}$ from a tree to
a linear order is \textit{increasing} if $\rho(s) \leq \rho(t)$
whenever $s \prec t$, and strictly so if the former inequality is
always strict. The next theorem is the key result of this section.

\begin{thm}
\label{yimpliesgru} If $\Upsilon$ is a tree and
$\mapping{\rho}{\Upsilon}{Y}$ is a strictly increasing function then
$\Upsilon$ is a Gruenhage space.
\end{thm}

Theorem \ref{dualrotund} and Theorem \ref{yimpliesgru}, together
with \cite[Proposition 7]{smith:05b} and a result from
\cite{smith:06}, allows us to present the following series of
equivalent conditions and, in particular, provides our partial
converse to Theorem \ref{dualrotund}. Observe that a locally compact
space is Gruenhage if and only if its 1-point compactification is.

\begin{cor}
If $\Upsilon$ is a tree then the following are equivalent:
\begin{enumerate}
\item $\dual{\Czerok{\Upsilon}}$ admits an equivalent, dual p-UR
norm;
\item $\dual{\Czerok{\Upsilon}}$ admits an equivalent, strictly convex dual
lattice norm;
\item $\Czerok{\Upsilon}$ admits an equivalent, G\^{a}teaux smooth
lattice norm;
\item $\dual{\Czerok{\Upsilon}}$ admits an equivalent, strictly convex dual
norm;
\item there is a strictly increasing function
$\mapping{\rho}{\Upsilon}{Y}$;
\item $\Upsilon$ is a Gruenhage space.
\end{enumerate}
\end{cor}

It is proved in \cite{smith:05b} that the 1-point compactification
of a tree $\Upsilon$ is descriptive, equivalently $\sigma$-discrete,
if and only if there is a strictly increasing function
$\mapping{\rho}{\Upsilon}{\rat}$. As trees go, those that admit such
$\rat$-valued functions are relatively simple. The order $Y$ is
considerably larger than $\rat$ in order-theoretic terms; indeed,
given any ordinal $\beta < \wone$, the lexicographic product
$\real^\beta$ embeds into $Y$. Accordingly, there is an abundance of
trees that admit strictly increasing $Y$-valued maps, but not
strictly increasing $\rat$-valued maps \cite{smith:05b}. Therefore,
the class of Gruenhage compact spaces encompasses appreciably more
structure than the class of descriptive compact spaces.

A little preparatory work must be presented before giving the proof
of Theorem \ref{yimpliesgru}. We recall some material from
\cite{smith:05b}.

\begin{defn}[\cite{smith:05b}]
\label{plateau} A subset $V \subseteq \Upsilon$ is called a
\textit{plateau} if $V$ has a least element $0_V$ and $V =
\bigcup_{t \in V}[0_V,t]$. A partition $\mathscr{P}$ of $\Upsilon$
consisting solely of plateaux is called a \textit{plateau
partition}.
\end{defn}

If $V$ is a plateau then $V\backslash\{0_V\}$ is open, so given a
plateau partition $\mathscr{P}$ of $\Upsilon$, the set $H =
\setcomp{0_V}{V \in \mathscr{P}}$ \textit{of least elements of} $V$
is closed in $\Upsilon$.

\begin{defn}[\cite{smith:05b}]
\label{admissiblepartitions} Given a tree $\Upsilon$, let
$(\mathscr{P}_\beta)_{\beta < \wone}$ be a sequence of plateau
partitions with the following properties:
\begin{enumerate}
\item if $\alpha < \beta$ and $V \in
\mathscr{P}_\alpha$, $W \in \mathscr{P}_\beta$, then either $W
\subseteq V$ or $V \cap W$ is empty;
\item if $\beta$ is a limit ordinal and $W
\in \mathscr{P}_\beta$, then
$$
W = \bigcap \setcomp{V}{V \in \mathscr{P}_\alpha, \alpha < \beta, W
\subseteq V};
$$
\item if $t \in \Upsilon$, there exists $\beta < \wone$, depending
on $t$, such that $\{t\} \in \mathscr{P}_\beta$.
\end{enumerate}
We call such a sequence of plateau partitions \textit{admissible}.
\end{defn}

\begin{defn}[\cite{smith:05b}]
\label{partitiontrees} Let $(\mathscr{P}_\beta)_{\beta < \wone}$ be
admissible and let $T$ be the tree
$$
\setcomp{(\alpha,V)}{V \in \mathscr{P}_\alpha, \alpha < \wone}
$$
with order $(\alpha,V) \prec (\beta,W)$ if and only if $\alpha \leq
\beta$ and $W \subseteq V$. Then the subtree
$$
\Upsilon(\mathscr{P})\;=\;\setcomp{(\beta,V) \in T}{U \mbox{ is not
a singleton whenever }(\alpha,U) \prec (\beta,V)}
$$
of $T$ is called the \textit{partition tree of }$\Upsilon$ with
respect to $(\mathscr{P}_\beta)_{\beta < \wone}$.
\end{defn}

It is evident that if $V$ is a plateau then so is $\closure{V}$,
with $0_{\closure{V}} = 0_V$. A subset of a tree $\Upsilon$ is
called an \textit{antichain} if it consists solely of pairwise
incomparable elements. With respect to the interval topology,
antichains are discrete subsets. We make the following elementary,
yet important, observation.

\begin{lem}
\label{disjoint} Let $E$ be an antichain in a partition tree
$\Upsilon(\mathscr{P})$. If $(\alpha,V)$ and $(\beta,W)$ are
distinct elements of $E$ then both intersections $V \cap W$ and
$\closure{V}\backslash\{0_V\} \cap \closure{W}\backslash\{0_W\}$ are
empty.
\end{lem}

\begin{proof}
We can assume that $\alpha \leq \beta$. That the first intersection
is empty follows directly from the definition of the partition tree
order. To see that the same is true for the second, note that if
$(\alpha,U) \preccurlyeq (\beta,W)$ then
$\closure{W}\backslash\{0_W\} \subseteq
\closure{U}\backslash\{0_U\}$, so all we need to do is prove that if
$t \in \closure{V}\backslash\{0_V\} \cap
\closure{U}\backslash\{0_U\}$ then $V$ and $U$ intersect
non-trivially and are thus equal. Given such $t$, we have that $0_V$
and $0_U$ are comparable. If $0_V \preccurlyeq 0_U$ then since there
exists $s \in (0_U,t] \cap V$, we have $0_U \in V$ as $V$ is a
plateau. Likewise, if $0_U \preccurlyeq 0_V$ then $0_V \in U$.
\end{proof}

The next result shows that if there is a strictly increasing
function $\mapping{\rho}{\Upsilon}{Y}$ then $\Upsilon$ admits a
partition tree $\Upsilon(\mathscr{P})$, on which may be defined a
strictly increasing, real-valued function. It is important to note
that the order of the partition tree is related to the order of
$\Upsilon$ through the second, albeit technical, property below. If
$t \in \Upsilon$ then the \textit{wedge} $[t,\infty)$ is the set
$\setcomp{u \in \Upsilon}{u \succcurlyeq t}$.

\begin{prop}[\cite{smith:05b}]
\label{yimpliespartitiontree} Let $\Upsilon$ be a tree. If
$\mapping{\rho}{\Upsilon}{Y}$ is strictly increasing then there
exists an admissible sequence of partitions
$(\mathscr{P}_\beta)_{\beta < \wone}$ that yields a partition tree
$\Upsilon(\mathscr{P})$, and a strictly increasing function
$\mapping{\pi}{\Upsilon(\mathscr{P})}{[0,1]}$. Moreover:
\begin{enumerate}
\item $\mathscr{P}_0 =  \setcomp{[r,\infty)}{r \in \Upsilon
\mbox{ is minimal}}$;
\item for any non-maximal $(\beta,V) \in \Upsilon(\mathscr{P})$, the map
$$
0_W \longmapsto \pi(\beta+1,W)
$$
is strictly decreasing on the subtree of least elements
$$
H_{(\beta,V)} = \setcomp{0_W}{(\beta+1,W) \in (\beta,V)^+}.
$$
\end{enumerate}
\end{prop}

In the proof below, we will assume the partition tree
$\Upsilon(\mathscr{P})$ and function $\pi$ from Proposition
\ref{yimpliespartitiontree}.

\begin{proof}[Proof of Theorem \ref{yimpliesgru}]
We construct a legitimate system on $\Upsilon$. As
$\Upsilon(\mathscr{P})$ admits a strictly increasing, real-valued
function $\pi$, its isolated elements may be decomposed into a
countable union of antichains $(F_n)$. Indeed, if $(\beta,W) \in
\Upsilon(\mathscr{P})$ is isolated and non-minimal, then it has an
immediate predecessor $(\alpha,V)$, and we can pick $\tau(\beta,W)
\in \rat \cap (\pi(\alpha,V), \pi(\beta,W))$. Then consider the
antichain of minimal elements, together with the fibres
$(\tau^{-1}(q))_{q \in \rat}$. If $V$ is a plateau then
$\closure{V}\backslash V$ is an antichain and hence discrete. Note
that here, closure is taken with respect to $\Upsilon$. From Lemma
\ref{disjoint}, the family
$\setcomp{\closure{V}\backslash\{0_V\}}{(\beta,V) \in F_n}$ is a
pairwise disjoint collection of open sets in $\Upsilon$. Hence $D_n
= \bigcup \setcomp{\closure{V}\backslash V}{(\beta,V) \in F_n}$ is
discrete.

Given $q \in \rat$, consider the set $E_q$ of successor elements
$(\beta+1,W) \in (\beta,V)^+$, with $(\beta,V) \in
\Upsilon(\mathscr{P})$ arbitrary, such that $\pi(\beta,V) < q <
\pi(\beta+1,W)$. Observe that $E_q$ is an antichain in
$\Upsilon(\mathscr{P})$. Indeed, if $(\alpha+1,U) \prec (\beta+1,W)$
and $(\beta+1,W) \in E_q$ then $(\alpha+1,U) \preccurlyeq (\beta,V)
\prec (\beta+1,W)$, thus $\pi(\alpha+1,U) \leq \pi(\beta,V) < q$. It
follows that $(\alpha+1,U) \notin E_q$. Given non-maximal $(\beta,V)
\in \Upsilon(\mathscr{P})$, property (2) of Proposition
\ref{yimpliespartitiontree} tells us that, in particular, the set of
relatively isolated points in the least elements $H_{(\beta,V)}$ can
be decomposed into a countable union of antichains
$(F_{(\beta,V),m})$ in $\Upsilon$. Given $(\beta+1,W) \in
(\beta,V)^+$ such that $0_W \in F_{(\beta,V),m}$, set
$$
E_{q,(\beta,V),W} \;=\; \setcomp{(\beta+1,W^\prime) \in E_q \cap
(\beta,V)^+}{0_W \preccurlyeq 0_{W^\prime}}
$$
and
$$
\mathscr{E}_{q,m} \;=\; \setcomp{E_{q,(\beta,V),W}}{(\beta+1,W) \in
(\beta,V)^+\mbox{ and } 0_W \in F_{(\beta,V),m}}.
$$
We observe that each $\mathscr{E}_{q,m}$ is a family of disjoint
subsets of $E_q$. Indeed, let $E_{q,(\beta,V),W},
E_{q,(\beta^\prime,V^\prime),W^\prime} \in \mathscr{E}_{q,m}$. If
$(\beta,V) \neq (\beta^\prime,V^\prime)$ then $(\beta,V)^+ \cap
(\beta^\prime,V^\prime)^+$ is empty and we are done, so we assume
that this is not the case. If $W \neq W^\prime$ then $0_W$ and
$0_{W^\prime}$ are incomparable in $\Upsilon$, so
$E_{q,(\beta,V),W}$ and $E_{q,(\beta^\prime,V^\prime),W^\prime}$
must be disjoint. By Lemma \ref{disjoint}, it follows that the sets
$$
J_{q,(\beta,V),W} = \bigcup \setcomp{W^\prime}{(\beta+1,W^\prime)
\in E_{q,(\beta,V),W}},
$$
$E_{q,(\beta,V),W} \in \mathscr{E}_{q,m}$, are also pairwise
disjoint.

We prove that $J = J_{q,(\beta,V),W}$ is a plateau. Evidently $0_W$
is the least element of $J$. Now suppose $t \in J$ and $0_W
\preccurlyeq s \preccurlyeq t$. We have to show that $s \in J$. As
$0_W,t \in V$ and $V$ is a plateau, $s \in V$ and so there exists
$(\beta+1,W^\prime) \in (\beta,V)^+$ such that $s \in W^\prime$. We
know that $t \in W^{\prime\prime}$, where
$(\beta+1,W^{\prime\prime}) \in E_q \cap (\beta,V)^+$ and $0_W
\preccurlyeq 0_W^{\prime\prime}$. Thus we have $0_W \preccurlyeq
0_{W^\prime} \preccurlyeq 0_{W^{\prime\prime}}$ and, by condition
(2) of Proposition \ref{yimpliespartitiontree},
$\pi(\beta+1,W^\prime) \geq \pi(\beta+1,W^{\prime\prime}) > q$. It
follows that $(\beta+1,W^\prime) \in E_q$ and $s \in J$.

At last, we have enough information to define our legitimate system.
Begin by setting $A = \Upsilon$ and $\mathscr{H} = \setcomp{\{t\}}{t
\in \Upsilon \mbox{ is isolated}}$. Then define $A_n =
\closure{D_n}$ and $\mathscr{H}_n = \setcomp{\{t\}}{t \in D_n}$.
Again using Lemma \ref{disjoint}, we are permitted to define
$A^\prime_n = \Upsilon$ and $\mathscr{H}_n^\prime =
\setcomp{V\backslash\{0_V\}}{(\beta,V) \in F_n}$. From the above
discussion, given $q \in \rat$ and $m \in \nat$, we can define
$A_{q,m} = \Upsilon$ and
$$
\mathscr{H}_{q,m} \;=\;
\setcomp{J_{q,(\beta,V),W}\backslash\{0_W\}}{E_{q,(\beta,V),W} \in
\mathscr{E}_{q,m}}.
$$

We claim that, together, the families $\mathscr{H}, \mathscr{H}_n,
\mathscr{H}_n^\prime$ and $\mathscr{H}_{q,m}$ separate points of
$\Upsilon$ in the manner of Proposition \ref{equivalences}, part
(2). Let $s,t$ be distinct elements of $\Upsilon$. If $s$ or $t$ is
an isolated point of $\Upsilon$, we can separate using
$\mathscr{H}$. Henceforth, we will assume that both $s$ and $t$ are
limit elements of $\Upsilon$. Let $V^s_\beta$ be the unique element
of $\mathscr{\beta}$ containing $s$, and likewise for $t$. Let
$\gamma < \wone$ be minimal, subject to the condition that
$V^s_\gamma \neq V^t_\gamma$. Such $\gamma$ exists by property (3)
of Definition \ref{admissiblepartitions}. By property (2) of
Definition \ref{admissiblepartitions}, $\gamma$ cannot be a limit
ordinal. If $\gamma = 0$ then $V = V^s_\gamma = [r,\infty)$ by
property (1) of Proposition \ref{yimpliespartitiontree}. Being
minimal in $\Upsilon$, $r$ is isolated, so $s \in V\backslash
\{0_V\}$. As $(0,V)$ is minimal in $\Upsilon(\mathscr{P})$, it is an
element of $F_n$ for some $n$. Consequently, we can separate $s$
from $t$ using $\mathscr{H}^\prime_n$.

We finish by tackling the case where $\gamma = \beta + 1$ for some
ordinal $\beta$. Let $W = V^s_{\beta+1}$ and $W^\prime =
V^t_{\beta+1}$. If $s \in W\backslash\{0_W\}$ then as $(\beta+1,W)$
is isolated in $\Upsilon(\mathscr{P})$, we can separate using some
$\mathscr{H}^\prime_n$ as above. We can argue similarly if $t \in
W^\prime\backslash \{0_{W^\prime}\}$ so, from now on, we assume that
$s = 0_W$ and $t = 0_{W^\prime}$, i.e.\ $s,t \in H_{(\beta,V)}$. If
$0_W$ is an immediate successor with respect to $H_{(\beta,V)}$,
i.e.\ if there exists $0_U \in H_{(\beta,V)}$ such that $0_U \prec
0_W$ and no element of $H_{(\beta,V)}$ lies strictly between the
two, then $0_W \in \closure{U}\backslash U$. Indeed, if $r \prec
0_W$ then as $0_W$ is a limit in $\Upsilon$, there exists $\xi \in
(\max\{r,0_U\},0_W]\backslash\{0_W\}$. Now $\xi$ must lie in $U$
because $0_U$ is the immediate predecessor of $0_W$ in
$H_{(\beta,V)}$. It follows that $0_W \in \closure{U}$ as required.
Now $(\beta+1,U)$ is in $F_n$ for some $n$, so $\{0_W\} \in
\mathscr{H}_n$, thus separating $0_W$ from $0_{W^\prime}$. As above,
we can argue similarly if $0_{W^\prime}$ is an immediate successor
with respect to $H_{(\beta,V)}$, so now we assume that neither $0_W$
nor $0_{W^\prime}$ are such elements. As $H_{(\beta,V)}$ has a least
element and is a Hausdorff tree in its own right, the greatest
element less than both $0_W$ and $0_{W^\prime}$ is some $0_U \in
H_{(\beta,V)}$ and, without loss of generality, we can assume that
$0_U \prec 0_W$. If $0_{U^\prime}$ is the immediate successor of
$0_U$ in $H_{(\beta,V)}$ then $0_{U^\prime} \prec 0_W$, because
$0_W$ is not such an element. Consequently, $0_{U^\prime} \in
F_{(\beta,V),m}$ for some $m$ so, given rational $q$ strictly
between $\pi(\beta,V)$ and $\pi(\beta+1,W)$, we have $0_W \in
J\backslash\{0_{U^\prime}\}$, where $J = J_{q,(\beta,V),U^\prime}$.
Since $0_{U^\prime} \not\preccurlyeq 0_{W^\prime}$ by maximality of
$0_U$, it follows that $J\backslash\{0_{U^\prime}\}$ separates $0_W$
from $0_{W^\prime}$.
\end{proof}

\section{Stability properties of Gruenhage spaces}

Our first stability property is purely topological.

\begin{thm}
\label{triumph} If $X$ is a Gruenhage space and $\mapping{f}{X}{Y}$
is a perfect, surjective mapping, then $Y$ is also Gruenhage.
\end{thm}

\begin{proof}
Let $X$ be a Gruenhage space and assume that we have families
$(\mathscr{U}_n)$ and sets $R_n$ satisfying Proposition
\ref{equivalences} (3). By adding new families
$\{\bigcup\mathscr{U}_n\}$ if necessary, we assume that given $n$,
there exists $m$ such that $R_m = \bigcup\mathscr{U}_n$. If $G
\subseteq \nat$ is finite, define
$$
\mathscr{V}_G \;=\; \textstyle{\setcomp{\bigcap_{i \in
G}U_i}{(U_i)_{i \in G} \in \prod_{i \in G}\mathscr{U}_i}}.
$$
Given a perfect, surjective map $\mapping{f}{X}{Y}$, we set
$$
\mathscr{V}_{F,G,k} \;=\; \textstyle{\setcomp{Y\backslash
f(X\backslash (\bigcup_{i \in F} R_i \cup
\bigcup\mathscr{F}))}{\mathscr{F} \subseteq \mathscr{V}_G \mbox{
and} \card{\mathscr{F}} = k}}
$$
for finite $F,G \subseteq \nat$ and $k \in \nat$. Since $f$ is
perfect, every element of $\mathscr{V}_{F,G,k}$ is open in $Y$.

Let $y,z \in Y$ be distinct. We show that there exists finite $F,G
\subseteq \nat$, $k \in \nat$ and $\mathscr{F} \subseteq
\mathscr{V}_G$ with cardinality $k$, such that
$$
\{y,z\} \cap \textstyle{Y\backslash f(X\backslash (\bigcup_{i \in F}
R_i \cup \bigcup\mathscr{F}))}
$$
is a singleton. Moreover, if $\mathscr{G} \subseteq \mathscr{V}_G$
has cardinality $k$ and
$$
y \in \textstyle{Y\backslash f(X\backslash (\bigcup_{i \in F} R_i
\cup \bigcup\mathscr{F}))}
$$
is non-empty, then $\mathscr{G} = \mathscr{F}$. From this, it
follows immediately that $Y$ is Gruenhage.

To prove this claim, we first construct a pair of decreasing
sequences of compact sets. Set $A_0 = f^{-1}(y)$ and $B_0 =
f^{-1}(z)$. Given $r \geq 0$, if, for all $n$, it is true that $(A_r
\cup B_r) \cap R_n = \varnothing$ or $A_r \cap B_r \subseteq R_n$,
then we stop. If not then let $n_{r+1}$ be minimal, subject to the
requirement that $(A_r \cup B_r) \cap R_{n_{r+1}} \neq \varnothing$
and $(A_r \cup B_r) \backslash R_{n_{r+1}} \neq \varnothing$. Put
$A_{r+1} = A_r\backslash R_{n_{r+1}}$ and $B_r\backslash
R_{n_{r+1}}$. Continuing in this way, either we stop at a finite
stage or continue indefinitely.

If the process stops at a finite stage $r \geq 0$, set $A = A_r$ and
$B = B_r$. Evidently $(A \cup B) \cap R_n = \varnothing$ or $A \cup
B \subseteq R_n$ for all $n$. If the process above continues
indefinitely, then we obtain a sequence $n_1 < n_2 < \ldots$ and
decreasing sequences $(A_i)$, $(B_i)$ of non-empty, compact sets.
Put $A = \bigcap_{i=0}^\infty A_i$ and $B = \bigcap_{i=0}^\infty
B_i$. Then, given any $n$, again we have $(A \cup B) \cap R_n =
\varnothing$ or $A \cup B \subseteq R_n$, lest we violate the
minimality of the first $n_i > n$.

If $A = \varnothing$ then by surjectivity, and compactness if
necessary, there is some $r \geq 1$ such that $A_r = \varnothing$.
Since $(A_{r-1} \cup B_{r-1})\backslash R_{n_r} \neq \varnothing$ by
construction, it is not the case that $B_r$ is empty, thus
$f^{-1}(y) \subseteq \bigcup_{i=1}^r R_{n_i}$ and $f^{-1}(z)
\not\subseteq \bigcup_{i=1}^r R_{n_i}$. Put $F = \{\oneton{n}{r}\}$
and let $G$ be arbitrary. Then $Y\backslash f(X\backslash \bigcup_{i
\in F} R_i)$ is the only element of $\mathscr{V}_{F,G,0}$ and
$$
\{y,z\} \cap \textstyle{Y\backslash f(X\backslash \bigcup_{i \in F}
R_i)} \;=\; \{y\}.
$$
If $B = \varnothing$ then we proceed similarly.

Now suppose that $A \neq \varnothing$ and $B \neq \varnothing$.
Define $K = A \cup B$ and let
$$
I \;=\; \textstyle{\setcomp{n \in \nat}{K \cap R_n =
\varnothing\mbox{ and }K \subseteq \bigcup\mathscr{U}_n}}.
$$
We have $K = \bigcup\setcomp{K \cap U}{U \in \mathscr{U}_n}$
whenever $n \in I$. Moreover, the sets in each $\setcomp{K \cap U}{U
\in \mathscr{U}_n}$, $n \in I$, are pairwise disjoint. In fact
slightly more can be said; if $x \in K \cap U \cap V$ for $U,V \in
\mathscr{U}_n$ and $n \in I$ then $U = V$. Indeed, if $x \in K \cap
U \cap V$, $U,V \in \mathscr{U}_n$ and $U \neq V$ then $x \in R_n$,
so $n \not\in I$.

Given distinct $a,b \in K$, there exists $n$ and $U \in
\mathscr{U}_n$ such that $\{a,b\} \cap U$ is a singleton. Firstly,
this means $n \in I$. Indeed, if $K \cap R_n \neq \varnothing$ then
$K \subseteq R_n$, meaning $a,b \in U$, which is not the case. Now
suppose $K \not\subseteq \bigcup\mathscr{U}_n$. We have
$\bigcup\mathscr{U}_n = R_m$ for some $m$, so $K \cap
\bigcup\mathscr{U}_n = K \cap R_m$ is empty, which again is not the
case. Thus $n \in I$. In particular, this means we can assume that
$\{a,b\} \cap U = \{a\}$ because $\setcomp{K \cap V}{V \in
\mathscr{U}_n}$ partitions $K$. By compactness, it follows that
there is finite $G \subseteq I$ and finite $\mathscr{E} \subseteq
\bigcup_{i \in G} \mathscr{U}_i$ such that
$$
A \;\subseteq\; \bigcup\setcomp{K \cap U}{U \in
\mathscr{E}}\quad\mbox{and}\quad B \;\not\subseteq\;
\bigcup\setcomp{K \cap U}{U \in \mathscr{E}}.
$$
Let $x \in K \cap U$, where $U \in \mathscr{E}$. For every $i \in
G$, we know from above that there is a unique $U_i \in
\mathscr{U}_i$ such that $x \in U_i$. By definition $\bigcap_{i \in
G} U_i \in \mathscr{V}_G$, and since $U \in \mathscr{U}_j$ for some
$j \in G$, we have $U = U_j$ and $x \in \bigcap_{i \in G} U_i
\subseteq U$. This allows us to take a finite subset $\mathscr{F}
\subseteq \mathscr{V}_G$, such that
$$
A \;\subseteq\; \bigcup\setcomp{K \cap V}{V \in
\mathscr{F}}\quad\mbox{and}\quad B \;\not\subseteq\;
\bigcup\setcomp{K \cap V}{V \in \mathscr{F}}.
$$
We choose $\mathscr{F}$ so that it has minimal cardinality $k$.

If necessary, we appeal to compactness to find $r \geq 0$ satisfying
$$
f^{-1}(y) \;\subseteq\; \bigcup_{i=1}^r R_{n_i} \;\cup\;
\bigcup\mathscr{F},
$$
$A \subseteq f^{-1}(y)\backslash \bigcup_{i=1}^r R_{n_i}$ and $B
\subseteq f^{-1}(z)\backslash \bigcup_{i=1}^r R_{n_i}$. Let $F =
\{\oneton{n}{r}\}$. Observe that if $\mathscr{G} \subseteq
\mathscr{V}_G$ and $f^{-1}(y) \subseteq \bigcup_{i \in F} R_i \cup
\bigcup\mathscr{G}$ then $A \subseteq \bigcup\mathscr{G}$, and
likewise for $f^{-1}(z)$ and $B$. Thus $f^{-1}(z) \not\subseteq
\bigcup_{i \in F} R_i \cup \bigcup\mathscr{F}$ and consequently
$$
\{y,z\} \cap \textstyle{Y\backslash f(X\backslash (\bigcup_{i \in F}
R_i \cup \bigcup\mathscr{F}))} \;=\; \{y\}.
$$
Now let $y \in Y\backslash f(X\backslash (\bigcup_{i \in F} R_i \cup
\bigcup\mathscr{G}))$, where $\mathscr{G} \subseteq \mathscr{V}_G$
has cardinality $k$. It follows that $A \subseteq
\bigcup\mathscr{G}$. We show that $\mathscr{G} = \mathscr{F}$. Take
$W \in \mathscr{F}$. By minimality of $k$
$$
A \;\not\subseteq\; \bigcup\setcomp{K \cap V}{V \in
\mathscr{F}\backslash\{W\}}
$$
thus there is $x \in A \cap W$. Take $V \in \mathscr{G}$ such that
$x \in V$. We claim that $W = V$. Indeed, $W = \bigcap_{i \in G}
W_i$ and $V = \bigcap_{i \in G} V_i$ for some $W_i,V_i \in
\mathscr{U}_i$, $i \in G$. Since $G \subseteq I$ and $x \in K \cap
W_i \cap V_i$, we have $W_i = V_i$ for all $i \in G$, hence $W = V
\in \mathscr{G}$. Therefore $\mathscr{F} \subseteq \mathscr{G}$ and,
by cardinality, we have equality as required.
\end{proof}

Next, something of a more functional analytic nature.

\begin{prop}
\label{gimpliesballg} If $K$ is a Gruenhage compact then so is
$\ball{\dual{\Ck{K}}}$.
\end{prop}

\begin{proof}
Let $(A_n,\mathscr{H}_n)$ be a legitimate system satisfying
properties (1) -- (3), presented after Corollary
\ref{desciptivegru}. We can and do assume that $\varnothing \in
\mathscr{H}_n$ for all $n$. Given $H \in \mathscr{H}_n$ and $q \in
(0,1) \cap \rat$, define the $\weakstar$-open set
$$
U^{(H,n,q)}_+ \;=\; \setcomp{\mu \in \ball{\dual{\Ck{K}}}}{\mu_+(H
\cup (K\backslash A_n)) > q}
$$
and let $\mathscr{U}^{(n,q)}_+ = \setcomp{U^{(H,n,q)}_+}{H \in
\mathscr{H}_n}$. Define $U^{(H,n,q)}_-$ and $\mathscr{U}^{(n,q)}_-$
in the corresponding manner. We claim that, with respect to
$\mathscr{U}^{(n,q)}_+$ and $\mathscr{U}^{(n,q)}_-$, $n \in \nat$
and $q \in (0,1) \cap \rat$, $\ball{\dual{\Ck{K}}}$ is a Gruenhage
compact in the sense of Definition \ref{gruenhage}.

Let $\mu,\nu \in \ball{\dual{\Ck{K}}}$ be distinct. Either $\mu_+
\neq \nu_+$ or $\mu_- \neq \nu_-$. We suppose that the former holds;
if the latter holds then we repeat the argument below using the sets
$U^{(H,n,q)}_-$ and $\mathscr{U}^{(n,q)}_-$. By Lemma \ref{norming},
there exists $n \in \nat$ and $H_0 \in \mathscr{H}_n$ such that
$\mu_+(H_0) \neq \nu_+(H_0)$. If $\mu_+(K\backslash A_n) \neq
\nu_+(K\backslash A_n)$ then set $H = \varnothing$. Otherwise, set
$H = H_0$. Either way, we have $\mu_+(H \cup (K\backslash A_n)) \neq
\nu_+(H \cup (K\backslash A_n))$ and, without loss of generality, we
suppose that $\mu_+(H \cup (K\backslash A_n)) < q < \nu_+(H \cup
(K\backslash A_n))$ for some rational $q$. Then $\{\mu,\nu\} \cap
U^{(H,n,q)}_+ = \{\nu\}$. Moreover, if $\mu \in
U^{(H^\prime,n,q)}_+$ for some $H^\prime \in \mathscr{H}_n$ then
$\mu_+(H^\prime) = \mu_+(H^\prime \cup (K\backslash A_n)) -
\mu_+(K\backslash A_n) > q - \mu_+(H \cup (K\backslash A_n)) > 0$.
Hence, as each $\mathscr{H}_n$ is a family of pairwise disjoint
sets, $\mu$ can only be in finitely many elements of
$\mathscr{U}^{(n,q)}_+$.
\end{proof}

We finish by using these two results to glean a further crop of
stability properties.

\begin{prop}
$\;$
\begin{enumerate}
\item If $K$ is a Gruenhage compact and $\mapping{\pi}{K}{M}$ is
continuous and surjective then $M$ is also Gruenhage;
\item if $X_n$, $n \in \nat$ are Gruenhage spaces then so is $\prod_n X_n$;
\item if $X$ is a Banach space, $F \subseteq \dual{X}$ is a separating subspace
and $K \subseteq X$ is a Gruenhage compact in the
$\sigma(X,F)$-topology then so is its symmetric,
$\sigma(X,F)$-closed convex hull.
\end{enumerate}
\end{prop}

\begin{proof}
(1) follows immediately from Theorem \ref{triumph}. To prove (2), we
let $X_n$ have a sequence $(\mathscr{U}_{n,m})_{m \in \nat}$ of
families of open sets satisfying Definition \ref{gruenhage}. It is
straightforward to verify that the families $(\mathscr{V}_{n,m})$,
defined by
$$
\mathscr{V}_{n,m} \;=\; \setcomp{\textstyle{\prod_{i < n}X_i \times
U \times \prod_{i > n} X_i}}{U \in \mathscr{U}_{n,m}}
$$
are witness to the fact that $\prod_{n \in \nat}X_n$ is Gruenhage.
To see that (3) holds, consider, as in Proposition
\ref{coarsegruenhage}, $K$ as a subset of $\dual{F}$ and the map $S$
restricted to $\ball{\dual{\Ck{K}}}$, which is Gruenhage by
Proposition \ref{gimpliesballg}. By (1), $S\ball{\dual{\Ck{K}}}
\subseteq \dual{F}$ is a Gruenhage compact in the
$\weakstar$-topology, giving (3).
\end{proof}

\bibliographystyle{amsplain}

\end{document}